\documentclass[11pt]{amsart}
\usepackage{amssymb}
\usepackage{amsmath}
\usepackage[all]{xy}
\usepackage{amsthm}
\usepackage{amscd} 
\newtheorem{thm}{Theorem}[section]
\newtheorem{rk}[thm]{Remark}
\newtheorem{prop}[thm]{Proposition}   
\newtheorem{clly}[thm]{Corollary}
\newtheorem{lemma}[thm]{Lemma}
\newtheorem{defi}[thm]{Definition}
\newtheorem{exam}[thm]{Example}

\newtheorem{nota}[thm]{Notation}

\newcommand{\bpf}{\begin{proof}}
\newcommand{\epf}{\end{proof}}
\newcommand{\bprop}{\begin{prop}}
\newcommand{\eprop}{\end{prop}}
\newcommand{\bthm}{\begin{thm}}
\newcommand{\ethm}{\end{thm}}
\newcommand{\brk}{\begin{rk}}
\newcommand{\erk}{\end{rk}}
\newcommand{\bdefi}{\begin{defi}}
\newcommand{\edefi}{\end{defi}}
\newcommand{\blemma}{\begin{lemma}}
\newcommand{\elemma}{\end{lemma}}
\newcommand{\bclly}{\begin{clly}}
\newcommand{\eclly}{\end{clly}}
\newcommand{\bnota}{\begin{nota}}
\newcommand{\enota}{\end{nota}}

\newcommand{\id}{\mbox{id}}

\newcommand{\cstar}{\mbox{$C^*$}}

\newcommand{\CC}{\mathbb{C}} 

\newcommand{\NN}{\mathbb{N}}

\newcommand{\ZZ}{\mathbb{Z}}

\newcommand{\cl}{\mathcal{L}}

\newcommand{\lan}{\langle}
\newcommand{\lran}{\rangle_{{\!\scriptscriptstyle{L}}}}
\newcommand{\rran}{\rangle_{{\!\scriptscriptstyle{R}}}}

\newcommand{\tzy}{{T^{\scriptscriptstyle{Z}}_y}}

\title{Toeplitz Matrices acting on the $\ell^2$-space of an imprimitivity bimodule}
\author{Beatriz Abadie} 
\address{Centro de Matem\'atica. Facultad de Ciencias. Igu\'a 4225, CP 11 400, Montevideo, Uruguay.}
\email{abadie@cmat.edu.uy}

\begin{document}

\begin{abstract}
 We give a definition of Toeplitz matrix acting on the $\ell^2$-space of an imprimitivity bimodule $X$ over a $C^*$-algebra $A$. 
 We characterize the set of Toeplitz matrices 
 as the closure in a certain topology of the image of the left regular representation of the crossed product $A\rtimes X$.
\end{abstract}

\maketitle

If $X$ is an imprimitivity bimodule over a \cstar-algebra $A$, the right Hilbert $A$-module $\ell^2(X)=\oplus_{n\in\ZZ} X^{\otimes n}$ provides a natural generalization of the Hilbert space 
$\ell^2(\ZZ).$
The main problem in generalizing the definition of Toeplitz matrix to this setting is that the $ij^{th}$ entry of the  matrix associated to an operator acting on $\ell^2(X)$ is an adjointable operator 
from $X^{\otimes j}$ to  
$X^{\otimes i}$, so one has to make sense of the condition $[T]_{ij}=[T]_{i-1\ j-1}$ characterizing classical Toeplitz matrices. 
However, it seems natural to identify the creation operators 
$T^j_\eta:X^{\otimes j}\rightarrow X^{\otimes i}$ and $T^{j-1}_\eta:X^{\otimes j-1}\rightarrow X^{\otimes i-1}$, given by tensoring by $\eta\in X^{\otimes i-j}$.
By making use of the notion of multiplier of an imprimitivity bimodule discussed by Echterhoff and Raeburn in \cite{er}, we show that there is a unique A-bimodule isomorphism from the set
of 
right adjointable maps
$\cl_R(X^{\otimes j}, X^{\otimes i})$ to $\cl_R(X^{\otimes j-1}, X^{\otimes i-1})$ that takes $T^j_\eta$  to $T^{j-1}_\eta$. This result enables us to define Toeplitz matrices acting on 
$\ell^2(X)$.

We then turn to the characterization of those operators on $\ell^2(X)$ that are associated to Toeplitz matrices. In order to do that, we consider the crossed product $A\rtimes X$ 
discussed in \cite {aee} and its left regular representation $\Lambda$ on $\ell^2(X)$, a canonical representation that agrees with the usual left regular 
representation for crossed products by an automorphism. Let  $\sigma$ be the initial topology  on $\cl(\ell^2(X))$  induced by the family of seminorms 
$\mathcal{F}=\{p_v:v\in X^{\otimes j}, j\in \ZZ\}$, where
$p_v(T)=\|T(v\delta_j)\|$.  Theorem \ref{tlrep} characterizes the set of Toeplitz matrices 
 as the $\sigma$-closure of the image of the left regular representation  on $\ell^2(X)$ of the crossed product $A\rtimes X$ defined in \cite{aee} .

\section{Preliminaries}
\label{prelim}

 In this section we briefly expose the background  on imprimitivity bimodules and their multipliers. 
 We refer the reader to, for instance,  
 \cite{rw}, \cite{la}, or \cite{irep} for further results and 
 constructions.

 Let $A$ be a \cstar-algebra. A {\em right inner product $A$-module} is a complex vector space $X$ that is a right $A$-module satisfying the condition
 \[\lambda (xa)=(\lambda x)a=x(\lambda a) \text{ for all } \lambda\in \CC, x\in X, \text{ and }a\in A,\]
 together with a pairing $\langle\ ,\ \rran:X\times X\longrightarrow A\text{ such that } $
 \begin{enumerate}
 \item $\langle x,\lambda y +\mu z\rran =\lambda \langle x,y \rran + \mu \langle x,z \rran,$
 \item $\langle x,ya\rran =\langle x,y\rran a,$
 \item  $\langle y, x\rran =  \langle x, y\rran ^*,$
 \item  $\langle x, x\rran \geq 0$ 
 \item $\langle x, x\rran = 0$ only if $x=0$. 
 
\end{enumerate}
 
 A right Hilbert $A$-module consists of an inner product $A$-module $X$ that is complete in the norm
 \[\|x\|=\|\langle x, x\rran \|^{1/2}.\]
 A right Hilbert $A$-module $X$ is said to be full if the ideal 
 \[\text{span}\{\langle x, y\rran:x,y\in X\}\]
 is dense in $A$.
 
 Let $X$ and $Y$ be right Hilbert $A$-modules. A function $T:X\longrightarrow Y$ is {\em adjointable} if there is a function $T^*:Y\longrightarrow X$ such that
 \[\langle Tx,y\rran=\langle x,T^*y\rran \text{ for all }x\in X, y\in Y.\]
 Adjointable maps turn out to be bounded linear $A$-module maps. Throughout this work, we will denote with  $\cl_R(X,Y)$ the set of right adjointable operators 
from $X$ to $Y$.

Left Hilbert $A$-modules are defined analogously, by considering left $A$-modules and by replacing conditions 1) and 2) above with
\begin{align*}
&1') \ \langle \lambda x + \mu y , z\lran =\lambda \langle x,z \lran + \mu \langle y,z \lran,\\
&2') \ \langle ax,y\lran =a\langle x,y\lran. 
\end{align*}
If $X$ and $Y$ are left Hilbert $A$-modules, left adjointable maps from $X$ to $Y$ are defined analogously to the right case. The set of left 
 adjointable operators from $X$ to $Y$ will be denoted by $\cl_L(X,Y)$.  

Let $A$ and $B$ be \cstar-algebras. An $A-B$ bimodule $X$ is an $A-B$ imprimitivity bimodule if it is a full left Hilbert $A$-module and a full right Hilbert $B$-module such that
\[\langle x, y \lran z= x\langle y, z\rran,\]
for all  $x,y,z\in X,$
There is no ambiguity regarding the norm in $X$ in that case, since (see, for instance, \cite[Remark 1.9]{bms}) 
\[\|\langle x,x\lran\|=\|\langle x,x \rran\|\]
for all $x\in X$.
\begin{rk}
As was shown in \cite[Remark 1.9]{bms}, if $X$ is an $A-B$ imprimitivity bimodule, then
\begin{equation}
 \label{adjointable}
  \langle xb, y \lran=\langle x, yb^* \lran \text{ and } \langle ax, y \rran=\langle x, a^*y \rran.
\end{equation}
for all  $x,y \in X, a\in A,$ and $b\in B$.
\end{rk}
First notice that, if $z\in X$, 
\begin{align*}
 \langle x,yb^*\lran z &= x\langle yb^*,z \rran\\
 &=x(\langle z,yb^*\rran)^*\\
 &=x(\langle z,y\rran b^*)^*\\
 &=xb\langle y,z\rran\\
 &=\langle xb,y\lran z.
\end{align*}
Thus, if $a=\langle x,yb^*\lran-\langle xb,y\lran$, then $az=0$ for all $z\in X$. 

Finally,
\[aa^*=a(\langle yb^*, x\lran-\langle y, xb\lran)=\langle ayb^*, x\lran-\langle ay, xb\lran=0,\]
which shows that $a=0$.
An analogous reasoning proves the second equation in (\ref{adjointable}).

The {\em dual of an $A-B$ imprimitivity bimodule} $X$ was defined in \cite[6.17]{irep}  as the $B-A$ imprimitivity bimodule $\tilde{X}$ consisting of  the conjugate vector space of $X$   
with the structure given by
\begin{gather*}
b \tilde{x}=\widetilde{xb^*}, \tilde{x} a= \widetilde{a^*x},\ 
\langle \tilde{x}, \tilde{y}\lran = \langle x,y\rran,\ \langle \tilde{x}, \tilde{y} \rran = \langle x,y\lran,
\end{gather*}
where $a\in A$, $b\in B$, and  $\tilde x$ denotes the element $x\in X$ viewed as an element of the dual bimodule ${\tilde X}$. 

Given an $A-B$ imprimitivity bimodule $X$ and a $B-C$ imprimitivity bimodule $Y$, the {\em tensor product} $X\otimes Y$ is the $A-C$ imprimitivity bimodule obtained by completing
 the left inner product $A$-module and right inner product $C$-module  consisting of the algebraic tensor product $X\otimes_{B\ alg} Y$ with  inner products given on simple tensors by
 \[\langle x_1\otimes y_1,  x_2\otimes y_2\lran= \langle x_1\langle y_1,y_2\lran, x_2\lran 
 \text{ and } \langle x_1\otimes y_1,  x_2\otimes y_2\rran=\langle y_1,\langle x_1,x_2\rran y_2\rran.\]

We now recall the notions of multiplier algebra and multiplier bimodule.
A {\em multiplier $m$ of a \cstar-algebra} $A$ is a pair $m=(L, R)$, where $L$ and $R$ are linear maps from $A$ to itself such that
\[L(ab)=L(a)b,\  R(ab)=aR(b),\text{ and } aL(b)=R(a)b,\]
for all $a,b\in A$. Every $a\in A$ gives rise to a multiplier $(L_a, R_a)$, where $L_a$ and $R_a$ are left and right multiplication by $a$, respectively.
The set $M(A)$ of multipliers of $A$ can be endowed with the structure of a \cstar-algebra \cite[2.1.5]{mu} in which $A$ sits as an ideal via the identification mentioned above. 
Besides (\cite[3.1.8]{mu}), 
if $B$ is a \cstar-algebra containing $A$ as an ideal, there is a unique homomorphism from $B$ to $M(A)$ that is the identity on $A$. 

In \cite[Definition 1.1]{er}, Echterhoff and Raeburn define the notion of the multiplier bimodule of an imprimitivity bimodule. A {\em multiplier of an $A-B$ imprimitivity bimodule} $Y$ 
consists of 
a pair $m=(m_A,m_B)$, where \mbox{$m_A:A\longrightarrow Y$} is A-linear, $m_B:B\longrightarrow Y$ is $B$-linear, and
\[m_A(a)b=am_B(b),\]
for all $a\in A$, $b\in B$.

Like in the case of \cstar-algebras, $y\in Y$ can be viewed as the multiplier $(y_A,y_B)$, where $y_A(a)=ay$ and $y_B(b)=yb$, for all $a\in A$ and $b\in B.$ By means of this identification,
the set $M(Y)$ of multipliers of $Y$ can be made into an $A-B$ bimodule, by setting
\[am=m_A(a)\text{ and } mb=m_B(b), \]
for all $a\in A$, $b\in B$, and $m=(m_A,m_B)\in M(Y).$

 Proposition 1.2 in \cite{er}  characterizes $M(Y)$ up to isomorphism  
as the $A-B$ bimodule satisfying the following two properties:
\begin{enumerate}
\item $M(Y)$ contains a copy of the bimodule $Y$ such that $AM(Y)\subset Y$ and $M(Y)B\subset Y$.
\item If $M$ is an $A-B$ bimodule satisfying (1), then there is a unique $A-B$ bimodule homomorphism from $M$ to $M(Y)$ that is the identity on $Y$.
\end{enumerate}
\section{Adjointable maps as multipliers}
\label{mbim}


 Let $Y$ and $Z$  be an $A-B$ and a $B-C$ imprimitivity bimodule, respectively. For $y_0\in Y$ and $z_0\in Z$,  we denote by $T^Z_{y_0}\in \cl_R(Z,Y\otimes Z)$ and 
 $R^Y_{z_0}\in\cl_L(Y,Y\otimes Z)$)
 the creation operators defined by
 \begin{equation}
 \label{creation}
 T^Z_{y_0}(z)=y_0\otimes z \text{ and } R^Y_{z_0}(y)=y\otimes z_0. 
  \end{equation}

It is well known, and easy to check, that the maps $y\mapsto T^Z_y$ and $z\mapsto R^Y_z$ are isometric, and that
 \[(T^Z_{y_0})^*(y\otimes z)=\langle y_0,y\rran z \quad \text{ and } \quad (R^Y_{z_0})^*(y\otimes z)=y\langle z,z_0\lran,\]
 for all $y,y_0\in Y$ and $z,z_0\in Z$.

 We will also be making use of the equation
 \begin{equation}
 \label{rstar}
  (R^Y_z)^*(\eta)\otimes w=\eta\langle z,w\rran,
 \end{equation}
for all $z,w\in Z$ and $\eta\in Y\otimes Z$.

By virtue of the continuity of both sides in Equation \ref{rstar}, it suffices to check it for $\eta$ in the algebraic tensor product  $Y\otimes_{alg} Z$. 

If $\eta=\sum_iy_i\otimes z_i$, then 
\begin{gather*}
(R^Y_{z})^*(\eta)\otimes w= \sum_iy_i\langle z_i,z\lran \otimes w=\sum_iy_i\otimes \langle z_i,z\lran w=\\
 =\sum_iy_i\otimes  z_i\langle z, w\rran=\eta \langle z, w\rran.
\end{gather*}

\begin{nota}
 \label{laabim}
 Let $Y$ and $Z$ be an $A-C$ and a $B-C$ imprimitivity bimodule, respectively. Throughout this work, we will view $\cl_R(Z,Y)$ as an $A-B$ bimodule for the actions
 \[(a\cdot\phi)(z)=a\phi(z)\qquad \text{ and }\qquad (\phi\cdot b)(z)=\phi(bz),\]
 for $a\in A$, $b\in B$, $z\in Z$, and $\phi\in \cl_R(Z,Y)$.
\end{nota}

Proposition 1.3 in \cite{er} identifies the multiplier bimodule $M(Y)$ of an $A-B$ imprimitivity bimodule $Y$ with $\mathcal{L}_R(B,Y)$, with the $A-B$  bimodule structure 
established in Notation \ref{laabim} and the copy of $Y$  obtained via the map 
$y\mapsto T^B_y\in \mathcal{L}_R(B,Y)$, where  $T_y^B(b)=yb$ for $b\in B$ and $y\in Y$. The following theorem generalizes that result, which follows when one takes $Z=B$.
\begin{thm}
 \label{mofy}
 Let $Y$ and $Z$  be an $A-B$ and a $B-B$ imprimitivity bimodule, respectively.
 Then the $A-B$ bimodule $\cl_R(Z,Y\otimes Z)$, provided with the copy of $Y$ given by $T_Y=\{\tzy:y\in Y\}$,   is isomorphic to $M(Y)$.
\end{thm}
\begin{proof}
 Let $M$ denote the $A-B$ bimodule $\cl_R(Z,Y\otimes Z)$. 
 First note that the map $y\rightarrow \tzy$  is an $A-B$ bimodule homomorphism: 
\[(T_y^Z\cdot b)(z)=T^Z_y(bz)=y\otimes bz=yb\otimes z=T^Z_{yb}(z),\]
and
\[(a\cdot T^Z_y)(z)=aT^Z_y(z)=a y\otimes z=T^Z_{ay}(z),\]
for all $a\in A$, $b\in B$, $y\in Y$, and $z\in Z$.

Therefore, we can identify the $A-B$ bimodule $Y$ 
 with the closed $A-B$ sub-bimodule $T_Y$ of $M$.

 We now show that $AM\subset T_Y$. Let $a=\langle u\otimes v,u'\otimes v'\rangle_L$, where $u,u'\in Y$ and $v,v'\in Z$.
 
 Then
 \begin{align*}
 (a\cdot\phi)(z)&= \langle u\otimes v, u'\otimes v'\lran\phi(z)\\
 &=u\otimes v\langle u'\otimes v',\phi(z)\rran\\
 &=u\otimes v\langle \phi^*(u'\otimes v'),z \rran\\
  &=u\otimes \langle v, \phi^*(u'\otimes v')\lran z\\
  &=T_{u\langle v,\phi^*(u'\otimes v')\lran}^{\scriptscriptstyle{Z}}(z),
\end{align*}
for all $z\in Z$. Therefore, $a\cdot\phi\in T_Y$. 

It follows that the set $A_0:=\{a\in A: a\cdot \phi\in T_Y\}$ is dense in $A$. Since 
$T_Y$ is closed in $M$ and the action of $A$ on $M$ is continuous, we conclude that $A_0=A$. 

We next show that $MB\subset T_Y$. Let $b=\langle u,v \lran$, with $u,v\in Z$, and let $\phi\in M$. If $z\in Z$, then
\[ (\phi\cdot b)(z)=\phi\big(\langle u,v \lran z\big)=\phi\big( u\langle v,  z\rran \big)=\phi(u)\langle v,z\rran.\] 
If $\phi(u)=\sum_{i=1}^n y_i\otimes z_i$, for $y_i\in Y$, $z_i\in Z$, $i=1,\cdots n$, then, by the equation above, 
\begin{equation}
\label{foreta}
(\phi\cdot b)(z) =\big(\displaystyle{\sum_{i=1}^n} y_i\otimes z_i\big)\langle v,z\rran=T^Z_{\sum_i y_i\langle z_i, v\lran}(z)=T^Z_{(R^Y_v)^*(\sum_i y_i\otimes z_i)}(z),
\end{equation}
where $R^Y_v$ is as in Equation (\ref{creation}).

We now show that 
\begin{equation}
 \label{phil}
\phi\cdot b=T^Z_{(R^Y_v)^*(\phi(u))}
\end{equation}
if $b=\langle u,v\lran$, for $u,v\in Z$.
Let  $\eta_k$ be a sequence in the algebraic tensor product $Y\otimes_{alg} Z$ converging to $\phi(u)$.
Then, as above,
\begin{align*}
 (\phi\cdot b)(z)&= \phi(u)\langle v,z\rran\\
 &=\lim_k\eta_k\langle v,z\rran\\
 &=\lim_k T^Z_{(R^Y_v)^*(\eta_k)}(z)\\
 &=T^Z_{(R^Y_v)^*(\phi(u))}(z).
\end{align*}

We have thus shown that $\phi\cdot \langle u,v \lran\in T_Y$, for all $u,v\in Z$. 
Now, a reasoning similar to that above shows that $(\phi\cdot b)\in T_Y$ for all $b\in B$.

The universal property of $M(Y)$ and the identification of $M(Y)$ with $\mathcal{L}_R(B,Y)$ mentioned above yield now an $A-B$ bimodule homomorphism
\[J:M\longrightarrow M(Y)\]
such that $J(\tzy)=T^B_y$ for all $y\in Y$.

Let  $H:M(Y)\longrightarrow M$ be defined by 
\[[H(\phi)](bz)=\phi(b)\otimes z,\]
for all $\phi\in M(Y)$, $b\in B$, and $z\in Z$. Notice that the definition above makes sense, since $H(\phi)$ is the composition of $\phi\otimes \id_Z$ and the canonical isomorphism between $Z$ and $B\otimes Z$.

Besides, $H$ is an $A-B$ bimodule homomorphism:
\[[H(a\cdot \phi)](bz)=(a\cdot \phi)(b)\otimes z=a \phi(b)\otimes z= [a\cdot H(\phi)](z), \]
and
\[[H(\phi\cdot c)](bz)=(\phi\cdot c)(b)\otimes z=\phi(cb)\otimes z=[H(\phi)\cdot c)](bz),\]
for all $\phi\in M(Y)$, $a\in A$, $b,c\in B$, and $z\in Z$. 

We now show that $H=J^{-1}$. In fact, we have that 
\[[H(T^B_y)](bz)=T^B_y(b)\otimes z=yb\otimes z=y\otimes bz=T^Z_y(bz),\]
for all $y\in Y$, $b\in B,$ and  $z\in Z$. That is, $H(T^B_y)=T^Z_y$ for all $y\in Y$. 

It now follows that $JH:M(Y)\longrightarrow M(Y)$ is an  $A-B$ bimodule homomorphism that is the 
identity on $Y$. We conclude from the universal property of $M(Y)$ that $JH=$Id$_{M(Y)}$.

Finally, we prove that $HJ=$Id$_M$.  First recall that, by  Equation (\ref{phil}), 
\[\phi\cdot b=\sum_{i=1}^n T^Z_{(R^Y_v)^*(\phi(u))},\] if $\phi\in M$ and $b=\langle u,v\lran$, where $u,v\in Z$.

Therefore,  if $c\in B$, then
\begin{equation}
 (J\phi)(bc)=[(J\phi)\cdot b](c)=[J(\phi\cdot b)](c)=\sum_{i=1}^n T^B_{(R^Y_v)^*(\phi(u))}(c).
\end{equation}
Then, by Equation (\ref{rstar}), for $c\in B$ and $z\in Z$,
\begin{align*}
(HJ\phi)(bcz)&=(J\phi)(bc)\otimes z\\
&= T^B_{(R^Y_v)^*(\phi(u))}(c)\otimes z\\
&= \phi(u)\lan v, cz\rran\\
&= \phi(u\lan v, cz\rran)\\
&=\phi(\langle u, v\lran cz)\\
&=\phi(bcz).
\end{align*}
for all $c\in B$ and $z\in Z$. 

A standard continuity argument completes now the proof.

\end{proof}

\section{Toeplitz matrices}
\label{toepmat}

Let $X$ be an imprimitivity bimodule over a $\cstar$-algebra $A$. In this section we make use of Theorem \ref{mofy} in order to define Toeplitz matrices acting on $\ell^2(X)$. We then describe 
Toeplitz matrices in terms of the left regular representation of the crossed product $A\rtimes X$ discussed in \cite{aee}.

 As usual, if $k<0$, $X^{\otimes k}$ denotes the  Hilbert $\cstar$-bimodule $\tilde{X}^{\otimes -k}$, $\tilde{X}$ 
being 
the dual bimodule defined in \cite{irep}. If $\eta\in X^{\otimes k}$, we denote by $T^n_\eta$ the operator $T^{X^{\otimes n}}_\eta\in \mathcal{L}_R(X^{\otimes n},X^{\otimes n+k})$, 
where we make the usual identifications of $a\otimes x$ with $ax$, $x\otimes a$ with $xa$, $\tilde{x}\otimes y$ with $\langle x,y\rran $, and $x\otimes \tilde{y}$ with $\langle x,y\lran $.

By Theorem \ref{mofy}, there is a unique  $A-A$ bimodule isomorphism 
\[\alpha^{n,m}:\cl_R(X^{\otimes n}, X^{\otimes m})\longrightarrow \cl_R(X^{\otimes n-1}, X^{\otimes m-1})\]
such that $\alpha^{n,m}(T^n_\eta)= T^{n-1}_\eta$, for all $n,m\in \ZZ$ and $\eta\in X^{\otimes m-n}.$

We denote by $\ell^2(X)$ the right Hilbert $\cstar$-module over $A$ given by
\[\ell^2(X)=\bigoplus_{k=-\infty}^{+\infty} X^{\otimes k},\]
and by $\cl(\ell^2(X))$ the space of right adjointable operators on $\ell^2(X)$.

An operator $T\in \cl(\ell^2(X))$ can be represented by a matrix $[T]$, where $[T]_{ij}\in \cl_R(X^{\otimes j}, X^{\otimes i})$ is given by
\[[T]_{ij}=\Pi_iTE_j,\]
for the usual maps
$E_k:X^{\otimes k} \longrightarrow \ell^2(X)$ and  $\Pi_k:\ell^2(X) \longrightarrow X^{\otimes k},$
defined by $E_k(u)=u\delta_k$ and $\Pi_k f=f(k)$, for all $k\in \ZZ$. 

The automorphisms $\alpha^{n,m}$ defined above yield a natural definition of Toeplitz matrix.

\begin{defi}
 Let $T\in \cl_R(\ell^2(X))$. The matrix $[T]$ is said to be a Toeplitz matrix if $\alpha^{j,i}([T]_{ij})=[T]_{i-1\ j-1}$ for all $i,j\in \ZZ$.
\end{defi}
\begin{exam} {\bf Classical Toeplitz matrices}
\end{exam}  Let $X=\CC$ be the $\CC-\CC$ imprimitivity  bimodule obtained by letting $\CC$ act on itself with left and right multiplication, with inner products
 \[\langle \lambda, \mu \lran=\lambda \overline{\mu} \text{ and } \langle \lambda, \mu \rran= \overline{\lambda}\mu.\]
 Since conjugation identifies the $\CC-\CC$ imprimitivity  bimodules $C$ and $\tilde{C}$, $\ell^2(X)$ is the usual Hilbert space $\ell^2(\ZZ)$.
 
 Besides, $\cl_R(X^{\otimes n}, X^{\otimes m})\simeq \cl_R(\CC, \CC)\simeq \CC$ consists of left multiplication by complex numbers, 
 and $\alpha^{n,m}=\id_\CC$ for all $n,m\in \ZZ$. It follows that the matrix $[T]$ associated to an operator $T\in \cl_R(\ell^2(X))$ is a Toeplitz matrix if and only if 
 $[T]_{ij}=[T]_{i-1\ j-1}$ for all $i,j\in \ZZ$. That is, if and only if $[T]$, viewed as a an operator acting on the Hilbert space $\ell^2(\ZZ)$, is a Toeplitz matrix in the classical 
 sense.

\begin{exam}
\label{cp}
Given an $A-A$ imprimitivity bimodule $X$, the crossed product $A\rtimes X$ was defined in \cite[Definition 2.4]{aee}. Theorem 2.9 in \cite{aee} shows that $A\rtimes X$ is
the cross-sectional \cstar-algebra of a Fell bundle  with fibers $\{X^{\otimes n}:n\in \ZZ \}$. It follows from \cite[VIII.16.12]{fd} that $A\rtimes X$ acts on $\ell^2(X)$ 
via the  representation   $\Lambda$ (the left regular representation, following the terminology of \cite[2.3]{amen})  given  by 
\[[\Lambda_f(\eta)](l)=\sum_{k\in \ZZ} f(l-k)\otimes \eta(k),\]
for all $\eta\in\ell^2(X)$, $l\in \ZZ$, and all compactly supported cross-sections  $f\in A\rtimes X$.

Therefore,
\[[\Lambda_f]_{ij}=T^j_{f(i-j)},\]
and $[\Lambda_f]$ is a Toeplitz matrix. 
\end{exam} 
\begin{thm}
\label{tlrep} Let $\sigma$ be the initial topology on  ${\mathcal L}(\ell^2(X))$ induced by the family of seminorms $\mathcal{F}=\{p_v:v\in X^{\otimes j}, j\in \ZZ\}$, where
$p_v(T)=\|T(v\delta_j)\|$, and let $\Lambda$ the left regular representation defined in Example \ref{cp}.

If $R\in {\mathcal B}(\ell^2(X))$, then  $[R]$ is a Toeplitz matrix if and only if  $R\in \overline{\Lambda(A\rtimes Z)}^\sigma$. 
\end{thm}
\begin{proof}
 Let $R\in \overline{\Lambda(A\rtimes Z)}^\sigma$.  Since the set of compactly supported cross-sections is dense in the norm in $A\rtimes X$, we may assume that $R$ is the 
 $\sigma$-limit of  a net $\{\Lambda_{f_d}\}$, where $f_d$ is a compactly supported cross-section 
 in $A\rtimes X$ for all $d$. We first assume that 
 $[R]_{ij}=T^j_{\eta_{ij}}$ for all $i,j\in \ZZ$ and  $\eta_{ij}\in X^{\otimes i-j} $.
 
 Then, if $v\in X^{\otimes j}$, 
 \[[R]_{ij}(v)=[R(v\delta_j)](i)=\lim_d[\Lambda_{f_d}(v\delta_j)](i)=\lim_d f_d(i-j)\otimes v.\]
 Therefore,
 \[\eta_{ij}\otimes v=\lim_d f_d(i-j)\otimes v,\]
for all $i,j\in \ZZ$ and $v\in X^{\otimes j}$.

Now, if $v\in X^{\otimes j}$ and  $w\in X$, then 
\[[R]_{i+1\ j+1}(v\otimes w)=\lim_{d} f_d(i-j)\otimes v\otimes w=\eta_{ij}\otimes v\otimes w=T^{j+1}_{\eta_{ij}}(v\otimes w). \]
  It follows that $\alpha^{j+1,i+1}([R]_{i+1\ j+1})=[R]_{ij}$ for all $i,j\in \ZZ$. Consequently, $[R]$ is a Toeplitz matrix.
  
  In the general case, since $R\cdot a$ is as above for all $a\in A$, then, for all $i,j\in \ZZ$,
  \[\alpha^{j,i}([R]_{ij}\cdot a)=\alpha^{j,i}([R\cdot a]_{ij})=[R\cdot a]_{i-1\ j-1}=[R]_{i-1\ j-1}\cdot a.\]
Since the maps $\alpha^{j,i}$ are $A-A$ bimodule homomorphisms, this shows that
\[\alpha^{j,i}([R]_{ij})\cdot a=[R]_{i-1\ j-1}\cdot a. \]
The result now follows from the fact that, if $\{e_{\lambda}\}$ is an approximate identity of $A$, then  $S(v)=\lim_{\lambda} (S\cdot e_{\lambda})(v)$ for all 
$S\in \mathcal{L}(X^{\otimes j}, X^{\otimes i})$, $v\in  X^{\otimes j}$, and $i,j\in \ZZ$.

We now turn to the converse statement. Let $[R]$ be a Toeplitz matrix. Assume first that $[R]$ is such that
\begin{equation}
\label{condab}
\text{ for all } k\in \ZZ \text{ there exists } u_k\in X^{\otimes k} \text{ such that } [R]_{ij}=T_{u(i-j)}^j.
\end{equation}

Set $u=\sum_k u_k\delta_k$. If $\{e_{\lambda}\}$ is an approximate identity of $A$, then, since $ue_{\lambda}=R(e_\lambda\delta_0)$,  $ue_\lambda\in \ell^2(X)$   and 
$\|ue_{\lambda}\|_2\leq \|R\|$ for all $\lambda$.
This implies that $u\in\ell^2(X)$.

Now, given $N\in \NN$, let $f_N\in A\rtimes Z$ be defined by $f_N=\sum_{|k|\leq N} u_k\delta _k$. We next show that $\Lambda_{f_N}$ converges to $R$ in the topology $\sigma$.

In fact, if $v\in X^{\otimes j}$, then 
\[(R-\Lambda_{f_N})(v\delta_j)=\sum_{|k|>N} (u_k\otimes v)\delta_{k+j}.\]
Therefore,
\begin{align*}
\|(R-\Lambda_{f_N})(v\delta_j)\|^2&=\|\sum_{|k|>N} \langle u_k\otimes v, u_k\otimes v\rran\|\\
&=\|\sum_{|k|>N} \langle  v, \langle u_k, u_k\rran v\rran\|\\
&=\|\langle  v, \sum_{|k|>N}\langle u_k, u_k\rran v\rran\|\\
&\leq \|\sum_{|k|>N} \langle u_k, u_k\rran\|\|v\|^2< \epsilon
\end{align*}
from some $N$ on.

For the general case, let  $\{e_\lambda\}$ be an approximate identity of $A$.  For each $\lambda$, $[R\cdot e_\lambda]$ is a Toeplitz matrix satisfying (\ref{condab}). 
Thus, for each $\lambda$, there is a sequence $\{f_{N,\lambda}\}$ of compactly 
supported functions in $A\rtimes X$  such that $\lim_N \Lambda_{f_{N,\lambda}}(v\delta_j)=R(e_\lambda v\delta_j)$, for all $j\in \ZZ$ and $v\in X^{\otimes j}$.

Given $\epsilon >0$ and $v_i\in X^{\otimes {j_i}}$ for $i=1,\cdots,k$, let $\lambda_0$ be such that \[\|v_i-e_{\lambda_0}v_i\|<\epsilon/(2\|R\|) \text{ for all }  i=1,\cdots,k. \]
Now choose $N_0$ 
so that
\[\|\Lambda_{f_{N_0,\lambda_0}}(v_i\delta_{j_i})-R(e_{\lambda_0}v_i\delta_{j_i})\|<\epsilon/2 \text{ for }  i=1,\cdots,k.\]
Then
\begin{align*}
\|(R-\Lambda_{f_{N_0,\lambda_0}})(v_i\delta_{j_i})\|&\leq \|R((v_i-e_{\lambda_0} v_i)\delta_{j_i})\|+\|R(e_{\lambda_0} v_i\delta_{j_i})-
\Lambda_{f_{N_0,\lambda_0}}(v_i\delta_{j_i})\|\\
&<\epsilon,
\end{align*}
for all  $i=1,\cdots,k$.
\end{proof}

 \end{document}